\documentclass[12pt]{article}
\usepackage{graphicx}
\usepackage{amsmath}
\usepackage{amssymb}
\usepackage{theorem}

     \usepackage{hyperref}

\usepackage{fancyhdr}

\numberwithin{equation}{section}

\newtheorem{thm}{Theorem}[section]
\newtheorem{lemma}[thm]{Lemma}
\newtheorem{prop}[thm]{Proposition}
\newtheorem{cor}[thm]{Corollary}
{\theorembodyfont{\rmfamily}

\newtheorem{rmk}[thm]{Remark}
}

\newcommand{\qed}{\hfill \mbox{\raggedright \rule{.07in}{.1in}}}
 
\newenvironment{proof}{\vspace{1ex}\noindent{\bf
Proof}\hspace{0.5em}}{\hfill\qed\vspace{1ex}}
\newenvironment{pfof}[1]{\vspace{1ex}\noindent{\bf Proof of
#1}\hspace{0.5em}}{\hfill\qed\vspace{1ex}}

\newcommand{\R}{{\mathbb R}}

\newcommand{\Z}{{\mathbb Z}}

\newcommand{\E}{{\mathbb E}}
\newcommand{\BBM}{{\mathbb M}}
\newcommand{\BBW}{{\mathbb W}}
\newcommand{\bfW}{{\mathbf W}}
\newcommand{\hBBM}{{\widehat{\mathbb M}}}

\newcommand{\x}{x^{(\epsilon)}}

\newcommand{\hx}{\hat x_{\epsilon}}

\newcommand{\cF}{{\mathcal F}}
\newcommand{\cG}{{\mathcal G}}
\newcommand{\cW}{{\mathcal W}}

\newcommand{\hv}{{\hat v}}
\newcommand{\hW}{{\widehat W}}
\newcommand{\tD}{{\widetilde D}}
\newcommand{\tM}{{\widetilde M}}

\newcommand{\eps}{\epsilon}

\newcommand{\Leb}{\operatorname{Leb}}

\newcommand{\diam}{\operatorname{diam}}

\newcommand{\SMALL}{\textstyle}

\title{Iterated invariance principle for slowly mixing dynamical systems}

\author{
Matt Galton\thanks{Mathematics Institute, University of Warwick, Coventry, CV4 7AL, UK}
\and 
Ian Melbourne\thanks{Mathematics Institute, University of Warwick, Coventry, CV4 7AL, UK}
}

\date{27 April 2020. Revised 3 March 2021.}

\begin{document}

 \maketitle

 \begin{abstract}
We give sufficient Gordin-type criteria for the iterated (enhanced) weak invariance principle to hold for deterministic dynamical systems.
Such an invariance principle is intrinsically related to the interpretation of stochastic integrals. We illustrate this with examples of deterministic fast-slow systems where our iterated invariance principle yields convergence to a stochastic differential equation.
 \end{abstract}

 \section{Introduction} 
 \label{sec-intro}

Recently, there has been a great deal of interest in homogenisation of deterministic systems with multiple timescales~\cite{CFKMZ19,CFKMZsub,Dolgopyat04,Dolgopyat05,GM13b,KM16,KM17,KKM18,KKMsub,MS11}; the aim is to prove convergence to a stochastic differential equation (SDE) as the separation of timescales increases.
The papers~\cite{GM13b,MS11} considered some simplified situations where it sufficed that the fast dynamics satisfies the weak invariance principle (WIP).
In general, however, there are issues regarding the correct interpretation of stochastic integrals (It\^o, Stratonovich, ...) in the limiting SDE that are not resolved by the WIP.  According to rough path theory~\cite{FrizHairer,FrizVictoir,Lyons98}, it is necessary to consider an iterated (or enhanced) WIP in order to determine the stochastic integrals.
Kelly \& Melbourne~\cite{KM16,KM17} applied rough path theory in the deterministic setting and reduced homogenisation theorems to establishing the iterated WIP and suitable moment control.  The conditions on moments were optimized in
Chevyrev {\em et al.}~\cite{CFKMZ19,CFKMZsub}.

The current paper is 
based on results of the first author in his Ph.\ D. thesis~\cite{GaltonPhD} and aims
to extend the class of dynamical systems for which the iterated WIP holds.  There is already a wealth of literature on the central limit theorem (CLT) and WIP for large classes of dynamical systems in both the dynamical systems and probability theory literature~\cite{DedeckerRio00,Gordin69,Heyde75,HofbauerKeller82,Keller80,Liverani96,MN05,TyranKaminska05,Volny93}.  We slightly extend the class of systems for which the WIP holds, and greatly extend the class of systems for which the iterated WIP holds.

Our approach is based on Gordin's method~\cite{Gordin69} for proving limit theorems via martingale approximation.
It is well-known that the $L^2$-criterion of Gordin~\cite{Gordin69} leads to the CLT and WIP, and it follows from~\cite{KM16} that the iterated WIP holds under this criterion (see~\cite[Proposition~2.5]{DMNapp}).
Proving the same results under the $L^1$ version of this criterion (hypotheses~\eqref{eq:H} and~\eqref{eq:Hinv} in this paper) is more delicate.
The CLT was obtained by~\cite{Gordin73} and much later the WIP was obtained by~\cite{DedeckerRio00}.  The WIP in~\cite{DedeckerRio00} is not quite in the right form for dynamical systems; in this paper we modify it so that it applies to dynamical systems by extending a time-reversal argument from~\cite{KM16}.

Previously there were no results on the iterated WIP under $L^p$ Gordin criteria for $p<2$ (except where there is additional Young tower structure, see~\cite{MV16} and~\cite[Section~10]{KM16}).  Addressing this is the main aim of this work.
In the noninvertible setting (Section~\ref{sec:Non}), we 
prove the iterated WIP under the $L^1$ Gordin criterion.
In the invertible setting (Section~\ref{sec:inv}), the validity of the iterated WIP under the $L^1$ Gordin criterion remains unresolved.  However, we prove the iterated WIP under a hybrid $L^1$--$L^2$ criterion~\eqref{eq:Hinv2} which is still a significant improvement on existing results.

The remainder of the paper is organised as follows. 
In Section~\ref{sec:Non}, we present our main results in the noninvertible setting.  Section~\ref{sec:inv} deals with the invertible setting.
In Section~\ref{sec:ex}, we consider some illustrative examples and in Section~\ref{sec:fs} we give an application to homogenisation of fast-slow systems.

\vspace{-2ex}
\paragraph{Notation}
For $a,b\in\R^d$, we define the outer product
$a\otimes b=ab^T\in\R^{d\times d}$.
For $J\in \R^{d\times d}$, we write
$|J|=\big(\sum_{i,j=1}^d J_{ij}^2\big)^{1/2}$.

For real-valued functions $f,\,g$, the integral $\int f\,dg$ denotes
the It\^o integral (where defined).
Similarly, for $\R^d$-valued functions,
$\int f\otimes dg$ denotes matrices of It\^o integrals.

We use ``big O'' and $\ll$ notation interchangeably, writing $a_n=O(b_n)$ or $a_n\ll b_n$
if there are constants $C>0$, $n_0\ge1$ such that
$a_n\le Cb_n$ for all $n\ge n_0$.

\section{Noninvertible setting}
\label{sec:Non}

Let $(\Lambda,\cF,\mu)$ be a probability space and 
$T:\Lambda\to\Lambda$ be an ergodic measure-preserving map.
Let $P:L^1\to L^1$ be the associated transfer operator (so
$\int_\Lambda Pv\,w\,d\mu=\int_\Lambda v\,w\circ T\,d\mu$
for $v\in L^1$, $w\in L^\infty$).
Also define the Koopman operator $Uv=v\circ T$.
We recall that 
\[
PU=I \quad\text{and} \quad  UP=\E(\cdot|T^{-1}\cF).
\]

Let $v\in L^\infty(\Lambda,\R^d)$ with $\int_\Lambda v\,d\mu=0$.
Our underlying hypothesis throughout this section is the
$L^1$ Gordin criterion
\begin{equation} \label{eq:H}
\sum_{n=1}^\infty |P^nv|_1<\infty.
\end{equation}
Under this hypothesis, it is well-known that the CLT and WIP hold.
We mention~\cite{Liverani96,DedeckerRio00} for this and related results.
Our aim is to prove the iterated (or enhanced) version of the WIP.
Previously, this was proved in~\cite[Section~4]{KM16} under the more restrictive assumption $\sum_{n=1}^\infty |P^nv|_2<\infty$.

Define the sequences of c\`adl\`ag processes
\[
W_n\in D([0,\infty),\R^d),\qquad
\BBW_n\in D([0,\infty),\R^{d\times d}),
\]
by
\begin{equation} \label{eq:W}
W_n(t)=\frac{1}{\sqrt n}\sum_{0\le j\le[nt]-1}v\circ T^j, \qquad
\BBW_n(t)=\frac1n\sum_{0\le i< j\le [nt]-1}(v\circ T^i)\otimes (v\circ T^j).
\end{equation}

\begin{thm} \label{thm:IWIP}
Let $v\in L^\infty(\Lambda,\R^d)$ with $\int_\Lambda v\,d\mu=0$,
and suppose that~\eqref{eq:H} holds.  Then 
\begin{itemize}
\item[(a)]
The limit
$\Sigma=\lim_{n\to\infty}\int_\Lambda W_n(1)\otimes W_n(1)\,d\mu\;\in \R^{d\times d}$
exists.
\item[(b)]
$\det\Sigma=0$ if and only if there exists $c\in\R^d$ nonzero 
and $h\in L^1$
such that $c\cdot v=h\circ T-h$.
\item[(c)] Let $\nu$ be any probability measure on $\Lambda$ absolutely continuous with respect to $\mu$ and regard 
$(W_n,\BBW_n)$ as a sequence of processes in $D([0,\infty),\R^d\times\R^{d\times d})$ on the probability space $(\Lambda,\nu)$.  

Then
$(W_n,\BBW_n)\to_w (W,\BBW)$ 
as $n\to\infty$, where $W$ is a $d$-dimensional Brownian motion with covariance $\Sigma$ and
\[
\BBW(t)=\int_0^t W \otimes dW+t\sum_{j=1}^\infty \int_\Lambda v\otimes(v\circ T^j)\,d\mu.
\]
\end{itemize}
\end{thm}

\begin{rmk} \label{rmk:GK}
A standard calculation using~\eqref{eq:H} and Theorem~\ref{thm:IWIP}(a) yields the Green-Kubo formula
\[
\Sigma=\int_\Lambda v\otimes v\,d\mu+\sum_{j=1}^\infty \int_\Lambda \big\{v\otimes (v\circ T^j)+(v\circ T^j)\otimes v\big\}\,d\mu.
\]
\end{rmk}

\begin{rmk}  The assumption that $T$ is noninvertible is not assumed explicitly in Theorem~\ref{thm:IWIP}, but hypothesis~\eqref{eq:H} implies that $v\equiv0$ when $T$ is invertible.  
\end{rmk}

\begin{rmk}
There are various possible extensions to Theorem~\ref{thm:IWIP}:
\\[.75ex]
(1) Let $1\le p,q\le\infty$ with $\frac1p+\frac1q=1$. Dedecker \& Rio~\cite{DedeckerRio00} consider unbounded functions $v\in L^p(\Lambda,\R)$ and prove that the ordinary WIP $W_n\to_w W$ holds provided the $1$-norm in~\eqref{eq:H} is replaced by the $q$-norm.
(In fact it suffices that $\sum_{n=1}^\infty P^nv$ converges in $L^q$ in~\cite{DedeckerRio00}.)
A natural question is to prove the iterated WIP $(W_n,\BBW_n)\to_w(W,\BBW)$ under such assumptions. However, the main motivation for studying the iterated WIP is its fundamental role in the theory of fast-slow systems (considered further in Section~\ref{sec:fs}) where it is standard to consider bounded $v$.
Also, considering unbounded $v$ would exacerbate the issues regarding
hypotheses~\eqref{eq:Hinv} and~\eqref{eq:Hinv2} in the invertible setting.
Hence we restrict in this paper to the case of bounded $v$.
\\[.75ex]
(2)
Dedecker \& Rio~\cite{DedeckerRio00} 
prove a nonergodic version of the WIP
following Voln\'y~\cite{Volny93}. It seems likely that a nonergodic version of the iterated WIP holds (paying due attention to the limit of $\BBW_n(1)-\BBM_n(1)$ in the proof of Theorem~\ref{thm:IWIP}(c)). Again, ergodicity of $\mu$ is a standard assumption in the motivating setting of fast-slow systems, and is assumed throughout this paper.
\\[.75ex]
(3) A third possible extension is to consider limits of $(W_n(s),\BBW_n(t))$ in
the space $D([0,\infty)\times[0,\infty),\R^d\times\R^{d\times d})$. This seems to involve a nontrivial extension of~\cite{JakubowskiMeminPages89,KurtzProtter91} and hence is beyond the scope of this paper.
\end{rmk}

Hypothesis~\eqref{eq:H} 
can be viewed as a slow mixing condition: we recall the following elementary result.
\begin{prop} \label{prop:H}
Let $v\in L^\infty(\Lambda,\R)$ with $\int_\Lambda v\,d\mu=0$.  
Suppose that there exists $a_n>0$ such that
\[
\Big|\int_\Lambda v\,w\circ T^n\,d\mu\Big|\le a_n|w|_\infty
\quad\text{for all $w\in L^\infty(\Lambda,\R)$, $n\ge1$.}
\]
Then
$|P^nv|_p\le |v|_\infty^{1-1/p}a_n^{1/p}$ for all $1\le p<\infty$. 
In particular, hypothesis~\eqref{eq:H} holds if $\sum_{n=1}^\infty a_n<\infty$.
\end{prop}

\begin{proof}  
See for example~\cite[Proposition~2.1]{MT02}.
\end{proof}

Throughout the remainder of this section, $L^p$ is shorthand for $L^p((\Lambda,\mu),\R^d)$ unless stated otherwise.

\subsection{Martingales}
Let $v:\Lambda\to\R^d$ be an $L^\infty$ observable with mean zero  satisfying
hypothesis~\eqref{eq:H}, and
define
\[
\chi^k_\ell=\sum_{j=\ell}^k P^jv,\,\; 1\le\ell\le k<\infty
\qquad\text{and} \qquad \chi=\sum_{j=1}^\infty P^jv.
\]
It follows from our assumptions that $\chi^k_\ell\in L^\infty$ for all $\ell\le k$ and $\chi\in L^1$.  Moreover, $\chi^k_1\to\chi$ in $L^1$ as $k\to\infty$.
Following~\cite{M09b}, we write
\begin{equation} \label{eq:mart}
v=m^{(k)}+\chi^k_1\circ T-\chi^k_1+P^kv,\,\;k\ge1 
\qquad\text{and}\qquad
v=m+\chi\circ T-\chi.
\end{equation}
Since $PU=I$, it is easily verified from the definitions in~\eqref{eq:mart} that $m,\,m^{(k)}\in \ker P$ for all $k$.
It is immediate that 
$m^{(k)}\in L^\infty$ for all $k$, that $m\in L^1$ and that $m^{(k)}\to m$ in $L^1$.  
A somewhat surprising fact due originally to~\cite{KipnisVaradhan86}, see also~\cite{DedeckerRio00,GordinPeligrad11,Liverani96,MaxwellWoodroofe00,TyranKaminska05}, is that
$m\in L^2$.  We begin by recovering this fact using an elementary argument.

\begin{lemma} \label{lem:L2}
$m\in L^2$ and $m^{(k)}\to m$ in $L^2$ as $k\to\infty$.
\end{lemma}

\begin{proof}
Working componentwise, we can suppose without loss that $d=1$.
For $\ell< k$,
\begin{align} \label{eq:L2} \nonumber
m^{(k)}-m^{(\ell)} & =(\chi_1^k-\chi_1^k\circ T-P^kv)
-(\chi_1^\ell-\chi_1^\ell\circ T-P^\ell v) \\
& = \chi_{\ell+1}^k-\chi_{\ell+1}^k\circ T+P^\ell v-P^kv
 = \chi_{\ell}^{k-1}-\chi_{\ell+1}^k\circ T.
\end{align}
Hence
\begin{align*}
|m^{(k)}-m^{(\ell)}|_2^2 & = 
\int_\Lambda (m^{(k)}-m^{(\ell)})(\chi_{\ell}^{k-1}-\chi_{\ell+1}^k\circ T)\,d\mu
=\int_\Lambda (m^{(k)}-m^{(\ell)})\chi_{\ell}^{k-1}\,d\mu,
\end{align*}
where we used that $m^{(k)},\,m^{(\ell)}\in\ker P$.
Continuing and using~\eqref{eq:L2} once more,
\begin{align*}
|m^{(k)}-m^{(\ell)}|_2^2 & = \int_\Lambda (\chi_{\ell}^{k-1}-\chi_{\ell+1}^k\circ T)\chi_{\ell}^{k-1}\,d\mu
=\int_\Lambda \{(\chi_{\ell}^{k-1})^2 -
\chi_{\ell+1}^k\,P\chi_{\ell}^{k-1}\}\,d\mu
\\ & =\int_\Lambda \{(\chi_{\ell}^{k-1})^2 -
(\chi_{\ell+1}^k)^2\}\,d\mu
 =\int_\Lambda (\chi_{\ell}^{k-1}- \chi_{\ell+1}^k)
 (\chi_{\ell}^{k-1}+ \chi_{\ell+1}^k)\,d\mu
 \\ & =\int_\Lambda (P^\ell v-P^kv)
 (\chi_{\ell}^{k-1}+ \chi_{\ell+1}^k)\,d\mu
\\ & \le (|P^\ell v|_\infty+|P^kv|_\infty)(|\chi_{\ell}^{k-1}|_1+ |\chi_{\ell+1}^k|_1)
\le 4|v|_\infty\sum_{n=\ell}^\infty|P^nv|_1.
\end{align*}
It follows from~\eqref{eq:H} that $m^{(k)}$ is Cauchy in $L^2$.
By uniqueness of limits in $L^1$, the $L^2$ limit of $m^{(k)}$ coincides with $m$.
\end{proof}

Elements of $\ker P$ enjoy the following martingale structure.

\begin{prop}   \label{prop:mart}
Let $\phi\in L^1\cap\ker P$ and
fix $n\ge1$.
Define $\cG_j=T^{-(n-j)}\cF$, $1\le j\le n$.
Then $\{\phi\circ T^{n-j},\,\cG_j;\,1\le j\le n\}$ is a sequence of
{\em martingale differences}.  That is, $\cG_1\subset\dots\subset \cG_n$,
$\phi\circ T^{n-j}$ is $\cG_j$-measurable for each $j$,
and $\E(\phi\circ T^{n-j}|\cG_{j-1})=0$ for each $j$.
\end{prop}

\begin{proof}
Since $T^{-1}\cF\subset\cF$, it follows that $\cG_j\subset\cG_{j+1}$.
Measurability of $\phi\circ T^{n-j}$ with respect to $\cG_j$ is clear.
Finally,
\begin{align*}
\E(\phi\circ T^{n-j}|\cG_{j-1})
& =\E(\phi|T^{-1}\cF)\circ T^{n-j}
 =(UP\phi)\circ T^{n-j}=0,
\end{align*}
since $\phi\in\ker P$.
\end{proof}

\subsection{Second moments}

Throughout, we write $v_n=\sum_{j=0}^{n-1}v\circ T^j$,
$m_n=\sum_{j=0}^{n-1}m\circ T^j$ and so on for observables
$v,\,m,\,\ldots$ defined on $\Lambda$.

\begin{cor} \label{cor:Doob}
Let $\phi\in L^2\cap\ker P$.  Then
$\big|\max_{1\le\ell\le n}|\phi_\ell|\big|_2\le 4\sqrt n|\phi|_2$
for all $n\ge1$.
In particular, $\big|\max_{1\le\ell\le n}|(m-m^{(k)})_\ell|\big|_2\le 4\sqrt n|m-m^{(k)}|_2$
for all $k,n\ge1$.
\end{cor}

\begin{proof}
Fix $n\ge1$ and let $X(j)=\phi\circ T^{n-j}$.
Since $\phi\in\ker P$, it follows that
$|X(1)+\dots+X(n)|_2 = \sqrt n|\phi|_2$.
By Proposition~\ref{prop:mart},
$\{X(j),\,\cG_j;\,1\le j\le n\}$ is a sequence of
martingale differences.  
Hence by Doob's inequality,
\[
\big|\max_{1\le\ell\le n}|X(1)+\dots+X(\ell)|\big|_2
\le 2|X(1)+\dots+X(n)|_2
= 2\sqrt n|\phi|_2.
\]
Finally, $\max_{1\le\ell\le n}|\phi_\ell|\le 2\max_{1\le\ell\le n}|X(1)+\dots+X(\ell)|$.
\end{proof}

Following~\cite{MN08}, we have a similar estimate for $v_n$.

\begin{prop} \label{prop:Rio}
$\big|\max_{1\le\ell\le n}|v_\ell|\big|_2^2\le 128\, n|v|_\infty\sum_{j=0}^\infty|P^jv|_1$.
\end{prop}

\begin{proof}
Fix $n\ge1$ and
define the random variables $X(j)=v\circ T^{n-j}$, $1\le j\le n$ which are adapted to the filtration $\cG_j=T^{-(n-j)}\cF$.
The version of Rio's inequality~\cite{Rio00} for $p=2$ 
in~\cite[Proposition~7]{MerlevedePeligradUtev06} states that
\[
\big|\max_{1\le\ell\le n}|X(1)+\cdots+X(\ell)|\big|_2^2\le 16\sum_{j=1}^n b_{j,n}
\]
where
\[
b_{j,n} =\max_{1\le j\le u\le n}|X(j)\sum_{k=j}^u\E(X(k)|\cG_j)|_1
\le |v|_\infty\max_{1\le j\le u\le n}\Big|\sum_{k=j}^u \E(v\circ T^{n-k}|\cG_j)\Big|_1.
\]
By Proposition~\ref{prop:mart},
$\E(m\circ T^{n-k}|\cG_j)=0$ for all
$k>j$.  By~\eqref{eq:mart},
\begin{align*}
\sum_{k=j}^u \E(v\circ T^{n-k}|\cG_j)
& =\E\big(m\circ T^{n-j}+ \chi\circ T^{n+1-j} -\chi\circ T^{n-u}|\cG_j\big)
\\ & =v\circ T^{n-j}+ \chi\circ T^{n-j} -\E\big(\chi\circ T^{n-u}|\cG_j\big).
\end{align*}
Hence
\[
b_{j,n}\le 
|v|_\infty(|v|_1+2|\chi|_1)\le 2|v|_\infty\sum_{j=0}^\infty|P^jv|_1
\]
and so
\[
\big|\max_{1\le\ell\le n}|X(1)+\cdots+X(\ell)|\big|_2^2\le 
32\, |v|_\infty\sum_{j=0}^\infty|P^jv|_1.
\]
Finally, $\max_{1\le\ell\le n}|v_\ell|\le 2\max_{1\le\ell\le n}|X(1)+\dots+X(\ell)|$.
\end{proof}

\begin{lemma} \label{lem:diff}
$\lim_{n\to\infty}\frac{1}{\sqrt n} \big|\max_{1\le\ell\le n}|(v-m)_\ell|\big|_2=0$.
\end{lemma}

\begin{proof}
By hypothesis~\eqref{eq:H} and Lemma~\ref{lem:L2}, for each $\eps>0$, there exists $k\ge1$ such that $\sum_{j=k}^\infty |P^jv|_1<\eps^2$ and
$|m-m^{(k)}|_2<\eps$.

Recall that $m,\,m^{(k)}\in\ker P$.
By Corollary~\ref{cor:Doob},
\begin{equation} \label{eq:diff1}
\frac{1}{\sqrt n}\big|\max_{1\le\ell\le n}|(m-m^{(k)})_\ell|\big|_2<4\eps.
\end{equation}

Next, $v=m^{(k)}+\chi_1^k\circ T-\chi_1^k+P^kv$, so
\begin{align*}
|(v-m^{(k)})_n|\le 
2|\chi_1^k|_\infty+|(P^kv)_n|
\le 2k|v|_\infty+|(P^kv)_n|.
\end{align*}
Note that $P^kv$ satisfies our underlying hypotheses, namely 
$P^kv\in L^\infty$, $\int_\Lambda P^kv\,d\mu=0$, 
$\sum_{n=1}^\infty |P^n(P^kv)|_1<\infty$.
Hence by Proposition~\ref{prop:Rio},
\begin{align*}
\big|\max_{1\le\ell\le n}|(v-m^{(k)})_\ell|\big|_2
& \le 
2k|v|_\infty+\big|\max_{1\le\ell\le n}|(P^kv)_\ell|\big|_2
\\ & \le 2k|v|_\infty + \Big\{128\,n|P^kv|_\infty\Big(\sum_{j=0}^\infty|P^{j+k}v|_1\Big)\Big\}^{1/2}
\\ & \le 2k|v|_\infty + \Big\{128\,n|v|_\infty\Big(\sum_{j=k}^\infty|P^jv|_1\Big)\Big\}^{1/2}
\ll k+\eps\sqrt n.
\end{align*}
Combining this with~\eqref{eq:diff1},
\(
\frac{1}{\sqrt n}\big|\max_{1\le\ell\le n}|(v-m)_\ell|\big|_2
 \ll \frac{1}{\sqrt n}k+\eps.
\)
Hence $\limsup_{n\to\infty} \frac{1}{\sqrt n}\big|\max_{1\le\ell\le n}|(v-m)_\ell|\big|_2\ll \eps$ and the result follows since $\eps$ is arbitrary.~
\end{proof}

\begin{pfof}{parts (a) and (b) of Theorem~\ref{thm:IWIP}}
Since $m\in\ker P$, it holds that
$\int_\Lambda m_n\otimes m_n\,d\mu=n\int_\Lambda m\otimes m\,d\mu$
for all $n$.
By Proposition~\ref{prop:Rio}, 
$|v_n|_2\ll n^{1/2}$.  Hence,
\begin{align*}
   \Big|
     n^{-1}\int_\Lambda v_n\otimes v_n \, & d\mu -    \int_\Lambda m\otimes m\,d\mu 
  \Big| 
 = n^{-1}  \Big|
     \int_\Lambda (v_n\otimes v_n -    m_n\otimes m_n)\,d\mu 
  \Big| 
\\ &  \le 
n^{-1}(|v_n|_2+|m_n|_2)|v_n-m_n|_2
\ll n^{-1/2}|v_n-m_n|_2\to0
\end{align*}
by Lemma~\ref{lem:diff}.
This proves part (a) and shows in addition that
\begin{equation} \label{eq:Sigma}
\Sigma =\int_\Lambda m\otimes m\,d\mu.
\end{equation}
It follows that
$c^T\Sigma c=\int_\Lambda (c\cdot m)^2\,d\mu$ for all $c\in\R^d$.

Next we prove part (b).
If $\det\Sigma=0$, then there exists $c\in\R^d$ nonzero such that $\Sigma c=0$ and hence 
$\int_\Lambda (c\cdot m)^2\,d\mu= c^T\Sigma c=0$,
so $c\cdot m=0$.
By~\eqref{eq:mart}, $c\cdot v=h\circ T-h$ where $h=c\cdot\chi\in L^1$.

Conversely, suppose that $c\cdot v=h\circ T-h$ for $c\in\R^d$
nonzero and $h\in L^1$.
Then $c\cdot Pv=h-Ph$.  Also,
$Pv=\chi-P\chi$, hence $c\cdot\chi-h\in\ker (P-I)$.
By ergodicity, $c\cdot\chi=h+aI$ for some $a\in\R$.
Substituting into~\eqref{eq:mart}, 
\[
c\cdot v = c\cdot m+c\cdot\chi\circ T-c\cdot\chi
=c\cdot m+h\circ T-h=c\cdot m+c\cdot v,
\]
and so $c\cdot m=0$.
Hence $c^T\Sigma c=\int_\Lambda (c\cdot m)^2\,d\mu=0$.
It follows that $\det\Sigma=0$.
\end{pfof}

\subsection{Iterated WIP}

In this subsection, we prove Theorem~\ref{thm:IWIP}(c).  First, we prove the ordinary WIP.

\begin{lemma} \label{lem:WIP}
$W_n\to_w W$ in $D([0,\infty),\R^d)$ as $n\to\infty$
on the probability space $(\Lambda,\mu)$.
\end{lemma}

\begin{proof}
It suffices to prove that $W_n\to_w W$ in $D([0,K],\R^d)$ for each fixed integer $K\ge1$.
Define
$M_n(t)=\frac{1}{\sqrt n}\sum_{0\le j\le [nt]-1}m\circ T^j$.
Recall that $m\in L^2\cap\ker P$.
By the pointwise ergodic theorem and~\eqref{eq:Sigma},
\[
n^{-1}\sum_{j=0}^{n-1}\{UP(m\otimes m)\}\circ T^j\to \int_\Lambda UP(m\otimes m)\,d\mu=\int_\Lambda m\otimes m\,d\mu=\Sigma\quad a.e.
\]
It follows from~\cite[Theorem~A.1]{KKM18} that
$M_n\to_w W$ in $D([0,K],\R^d)$.
Also,
\[
\sup_{t\in[0,K]}|W_n(t)-M_n(t)|=
\frac{1}{\sqrt n}\max_{1\le\ell\le nK}|(v-m)_\ell|\to_p 0
\]
by Lemma~\ref{lem:diff}.
Hence $W_n\to_w W$ in $D([0,K],\R^d)$.
\end{proof}

Define the sequence of processes
\[
\BBM_n\in D([0,\infty),\R^{d\times d}), \qquad
\BBM_n(t)=\frac1n\sum_{0\le i< j\le [nt]-1}(m\circ T^i)\otimes (v\circ T^j).
\]

\begin{lemma} \label{lem:BBM}
$(W_n,\BBM_n)\to_w (W,\BBM)$ in $D([0,\infty),\R^d\times\R^{d\times d})$ as $n\to\infty$ 
on the probability space $(\Lambda,\mu)$,
where $\BBM(t)=\int_0^t W\otimes dW$.
\end{lemma}

\begin{proof}
It suffices to prove that $(W_n,\BBM_n)\to_w (W,\BBM)$ in $D([0,K],\R^d\times\R^{d\times d})$ for each fixed integer $K\ge1$.
Define for $t\in[0,K]$,
\begin{align} \label{eq:WM-}
W_n^-(t) & =\frac{1}{\sqrt n}\sum_{1\le j\le [nt]}v\circ T^{nK-j}, \qquad 
M_n^-(t)=\frac{1}{\sqrt n}\sum_{1\le j\le [nt]}m\circ T^{nK-j}, \\
\BBM_n^-(t) & =\frac1n\sum_{1\le i<j\le [nt]}(v\circ T^{nK-i})\otimes (m\circ T^{nK-j}).
\nonumber
\end{align}

There are three main steps:
\begin{description}
\item[Step 1] Transfer convergence of $W_n$ in Lemma~\ref{lem:WIP} to 
convergence of $W_n^-$ and $M_n^-$, showing that
$(W_n^-,M_n^-)\to_w(W,W)$ in $D([0,K],\R^d\times\R^d)$.
\item[Step 2] Apply~\cite{JakubowskiMeminPages89,KurtzProtter91} to show that
$(W_n^-,M_n^-,\BBM_n^-)\to_w (W,W,\BBM)$
in $D([0,K],\R^d\times\R^d\times\R^{d\times d})$.
\item[Step 3] Transfer convergence of $(W_n^-,M_n^-,\BBM_n^-)$ in Step~2 back to
convergence of $(W_n,\BBM_n)$, yielding the desired result.
\end{description}

Let $\tD$ denote c\`agl\`ad functions.
Following~\cite{KM16}, we define
\[
g:D([0,K],\R^d)\to \tD([0,K],\R^d), \qquad
g(r)(t)=r(K)-r(K-t).
\]
Then
\begin{align*}
W_n(t) & 
=\frac{1}{\sqrt n}\sum_{j=nK-[nt]+1}^{nK}v\circ T^{nK-j} 
  = 
\frac{1}{\sqrt n}\sum_{j=[n(K-t)]+1}^{nK}v\circ T^{nK-j}-F_n^1(t)
\\ &  =W_n^-(K)-W_n^-(K-t)-F_n^1(t)
  =g(W_n^-)(t)-F_n^1(t),
\end{align*}
where $F_n^1(t)$ is either $0$ or $n^{-1/2}v\circ T^{nK-[n(K-t)]-1}$.
In particular,
\[
  \sup_{t \in[0,K]}|F^1_n(t)| \le n^{-1/2}|v|_\infty\to 0.
\]
By Lemma~\ref{lem:WIP} and the continuous mapping theorem,
\[
W_n^-=g^{-1}(W_n+F_n^1)\to_w g^{-1}(W)
\quad\text{ in $\tD([0,K],\R^d)$.}
\]
Using the fact that
the limiting process has continuous sample paths, it follows (see~\cite[Proposition~4.9]{KM16}) that
$W_n^-\to_w g^{-1}(W)$ in $D([0,K],\R^d)$.
By~\cite[Lemma~4.11]{KM16},
 the processes $g^{-1}(W)$ and $W$ are equal in distribution, so
$W_n^-\to_w W$ in 
$D([0,K],\R^d)$.
By the continuous mapping theorem, 
$(W_n^-,W_n^-)\to_w(W,W)$ in $D([0,K],\R^d\times\R^d)$.
Also, 
\begin{equation} \label{eq:diff-}
\sup_{t\in[0,K]}|W_n^-(t)-M_n^-(t)|\le 2n^{-1/2}\max_{1\le\ell\le nK}
|(v-m)_\ell|
\end{equation}
so $\big|\sup_{t\in[0,K]}|W_n^-(t)-M_n^-(t)|\big|_2\to0$ by Lemma~\ref{lem:diff}.
Hence 
$(W_n^-,M_n^-)\to_w(W,W)$ in $D([0,K],\R^d\times\R^d)$ completing Step~1.

\vspace{1ex}
By Proposition~\ref{prop:mart}, 
$\{m\circ T^{nK-j};\,1\le j\le nK\}$ is a martingale difference sequence with respect to the filtration
$\cG_{n,j}=T^{-(nK-j)}\cF$ for each $n\ge1$.
Moreover, $W_n^-$ is adapted (i.e.\ $v\circ T^{nK-j}$ is $\cG_{n,j}$-measurable for all $j,n$).
Also
$\int_{\Lambda}|M_n^-(t)|^2\,d\mu=n^{-1}[nt]\int_{\Lambda}|m|^2\,d\mu
\le K|m|_2^2$, so
condition~C2.2(i) in~\cite[Theorem~2.2]{KurtzProtter91} is satisfied.
Applying~\cite[Theorem~2.2]{KurtzProtter91} (or alternatively~\cite{JakubowskiMeminPages89}) we deduce that
$(W_n^-,M_n^-,\BBM_n^-)\to_w (W,W,\BBM)$
in $D([0,K],\R^d\times\R^d\times\R^{d\times d})$ completing Step~2.

\vspace{1ex}
Adapting~\cite{KM16}, we define
$h:D([0,K],\R^d\times\R^d\times\R^{d\times d})\to
\tD([0,K],\R^d\times\R^{d\times d})$,
\[
h(r,u,v)(t)=\big(r(K)-r(K-t)\,,\,\{v(K)-v(K-t)-r(K-t)\otimes (u(K)-u(K-t))\}^*\big),
\]
where ${}^*$ denotes matrix transpose.

We claim that
\[
\SMALL (W_n,\BBM_n)=h(W_n^-,M_n^-,\BBM_n^-)-F_n\quad\text{where}\quad
\sup_{t\in[0,K]}|F_n(t)|\to_p 0.
\]
Suppose that the claim is true.
By the continuous mapping theorem and~\cite[Proposition~4.9]{KM16},
$(W_n,\BBM_n)\to_w h(W,W,\BBM)$ in $D([0,K],\R^d\times\R^{d\times d})$.
By~\cite[Lemma~4.11]{KM16},
 the processes $h(W,W,\BBM)$ and $(W,\BBM)$ are equal in distribution so
$(W_n,\BBM_n)\to_w (W,\BBM)$ in $D([0,K],\R^d\times\R^{d\times d})$.

It remains to prove the claim.  Write $h=(h^1,h^2)$ where
$h^1:D([0,K],\R^d\times\R^d\times\R^{d\times d})\to
\tD([0,K],\R^d)$ and
$h^2:D([0,K],\R^d\times\R^d\times\R^{d\times d})\to
\tD([0,K],\R^{d\times d})$.

By Step~1,
\[
W_n(t) 
=h^1(W_n^-,M_n^-,\BBM_n^-)(t)-F_n^1(t)
\quad\text{where}\quad
  \sup_{t \in[0,K]}|F^1_n(t)| \le n^{-1/2}|v|_\infty\to 0.
\]
Also,
\begin{align*}
\BBM_n(t) & =
\frac1n\sum_{nK-[nt]< j<i\le nK}(m\circ T^{nK-i})\otimes (v\circ T^{nK-j}) \\
& =
\frac1n\sum_{nK-[nt]< i<j\le nK}\{(v\circ T^{nK-i})\otimes (m\circ T^{nK-j})\}^* 
\\ & =\frac1n\sum_{[n(K-t)]< i<j\le nK}\{(v\circ T^{nK-i})\otimes (m\circ T^{nK-j})\}^* -F^2_n(t)^*
 \\ & = \{\BBM_n^-(K)-\BBM_n^-(K-t)-W_n^-(K-t)\otimes(M_n^-(K)-M_n^-(K-t))\}^*-F^2_n(t)^*
 \\ & = h^2(W_n^-,M_n^-,\BBM_n^-)(t)-F^2_n(t)^*,
\end{align*}
where
$F^2_n(t)$ is either $0$ or
$n^{-1}\sum_{[n(K-t)] + 1 < j \le nK} (v\circ T^{nK-[n(K-t)]-1}) \otimes (m\circ T^{nK-j})$.
In particular,
  $|F_n^2(t)| \le n^{-1} |v|_\infty \, \max_{1\le\ell\le nK} |m_\ell|$, so
by Corollary~\ref{cor:Doob},
\[
  \big|\sup_{t\in[0,K]}|F_n^2(t)|\big|_2 \ll n^{-1/2} |v|_\infty \, |m|_2\to0.
\]
This completes the proof of the claim.
\end{proof}

\begin{pfof}{Theorem~\ref{thm:IWIP}(c)}
First we consider the case $\nu=\mu$.
It follows from the definition of $\chi$ that
\[
\int_\Lambda \chi\otimes v\,d\mu=
\sum_{j=1}^\infty \int_\Lambda (P^jv)\otimes v\,d\mu
=\sum_{j=1}^\infty \int_\Lambda v\otimes (v\circ T^j)\,d\mu.
\]
By Lemma~\ref{lem:BBM}, it suffices to show for all $K>0$ that
\[
\Big|\sup_{t\in[0,K]}
\Big(\BBW_n(t)-\BBM_n(t)-t\int_\Lambda \chi\otimes v\,d\mu 
\Big)\Big|\to_p0 \quad\text{as $n \to\infty$.}
\]

Now,
\begin{align*}
\sum_{j=1}^n\sum_{i=0}^{j-1}((v-m)\circ T^i)\otimes(v\circ T^j)
& =\sum_{j=1}^n(\chi\circ T^j-\chi)\otimes(v\circ T^j)
\\ & =\sum_{j=1}^n (\chi\otimes v)\circ T^j
-\chi\otimes\sum_{j=1}^n v\circ T^j.
\end{align*}
Since $v\in L^\infty$, $\chi\in L^1$, and $\int_\Lambda v\,d\mu=0$, it follows from the pointwise ergodic theorem that 
\[
\BBW_n(1)-\BBM_n(1)  =n^{-1}\sum_{j=1}^n\sum_{i=0}^{j-1}((v-m)\circ T^i)\otimes(v\circ T^j)\to \int_\Lambda \chi\otimes v\,d\mu \quad\text{a.e.}
\]
as $n\to\infty$.
Hence for any $K>0$, 
\[
\Big|\sup_{t\in[0,K]}
\Big(\BBW_n(t)-\BBM_n(t)-t\int_\Lambda \chi\otimes v\,d\mu 
\Big)\Big|\to0 \quad\text{a.e.}
\]
The iterated WIP on $(\Lambda,\mu)$ follows.

Now we consider the case where $\nu$ is a general probability measure absolutely continuous with respect to $\mu$.
Since $\mu$ is ergodic, it suffices by~\cite[Theorem~1]{Zweimuller07} to show that
\begin{align} \label{eq:bfW}
\lim_{n\to\infty} \mu\Big(\sup_{t\in[0,K]}|{\bfW}_n(t)\circ T-{\bfW}_n(t)|>\eps\Big)=0\end{align}
for all $\eps>0$, where ${\bfW}_n=(W_n,\BBW_n)$.

Now, 
$W_n(t)\circ T-W_n(t)=n^{-1/2}(v\circ T^{[nt]}-v)$ so
\[
|W_n(t)\circ T-W_n(t)|\le 2n^{-1/2}\max_{0\le k\le nK}|v\circ T^k|\le 2n^{-1/2}|v|_\infty
\]
for all $t\in[0,K]$.
Similarly,
\[
|\BBW_n(t)\circ T-\BBW_n(t)|\le 
2n^{-1}|v|_\infty \max_{1\le k\le nK}|v_k|
\]
for all $t\in[0,K]$.
By Proposition~\ref{prop:Rio} and~\eqref{eq:H},
$\big|\max_{1\le k\le nK}|v_k|\big|_2\ll \Big(n|v|_\infty \sum_{j\ge0}|P^jv|_1\Big)^{1/2} \ll n^{1/2}$.  Hence
$\big|\sup_{t\in[0,K]}|\bfW_n(t)\circ T-\bfW_n(t)|\big|_2\ll n^{-1/2}$, and~\eqref{eq:bfW} follows.
\end{pfof}

\section{Invertible setting}
\label{sec:inv}

Let $(\Lambda,\cF,\mu)$ be a probability space and 
$T:\Lambda\to\Lambda$ be an invertible ergodic measure-preserving map.
We suppose that there is a sub-sigma-algebra $\cF_0\subset \cF$ such that
$T^{-1}\cF_0\subset \cF_0$.
Then $\cF_j=T^j\cF_0$ defines a nondecreasing filtration $\{\cF_j:j\in\Z\}$.

Fix $d\ge1$ and let $v\in L^\infty(\Lambda,\R^d)$ with $\int_\Lambda v\,d\mu=0$.
The $L^1$ Gordin criterion now takes the form
\begin{equation} \label{eq:Hinv}
\sum_{n=1}^\infty |\E_0(v\circ T^{-n})|_1+
\sum_{n=0}^\infty |\E_0(v\circ T^n)-v\circ T^n|_1<\infty,
\end{equation}
where $\E_j=\E(\,\cdot\,|\cF_j)$.

Under hypotheses similar to~\eqref{eq:Hinv}, the CLT and WIP have been proved by various authors, including~\cite{DedeckerRio00,PeligradUtev05,TyranKaminska05}.
In Subsection~\ref{sec:WIPinv}, we recover the WIP under hypothesis~\eqref{eq:Hinv} using techniques similar to those in Section~\ref{sec:Non} combined with ideas from~\cite{DedeckerRio00}.

The iterated WIP holds under the $L^2$ Gordin criterion
$\sum_{n=1}^\infty |\E_0(v\circ T^{-n})|_2+
\sum_{n=0}^\infty |\E_0(v\circ T^n)-v\circ T^n|_2<\infty$ 
by~\cite[Section~4]{KM16} (see~\cite[Proposition~2.5]{DMNapp}).
An interesting open question is to prove the iterated WIP under the $L^1$ criterion~\eqref{eq:Hinv}, but this seems currently out of reach.
In Subsection~\ref{sec:IWIPinv}, we prove the iterated WIP under a hybrid $L^1$--$L^2$ Gordin criterion
\begin{equation} \label{eq:Hinv2}
\sum_{n=1}^\infty |\E_0(v\circ T^{-n})|_1+
\sum_{n=0}^\infty |\E_0(v\circ T^n)-v\circ T^n|_2<\infty.
\end{equation}
The same argument works if the roles of 
$|\;|_1$ and $|\;|_2$ are reversed in~\eqref{eq:Hinv2}.

We note that the existence of a suitable sub-sigma-algebra $\cF_0$ is very natural in the dynamical setting.   Indeed it is often the case that $\Lambda$ is covered by a collection $\cW^s$ of disjoint measurable sets, called ``stable 
leaves'', such that $TW^s_x\subset W^s_{Tx}$ for all $x\in\Lambda$, where $W^s_x$ is the stable leaf containing $x$.
In this situation, let $\cF_0$ denote the sigma-algebra generated by $\cW^s$.  Then $T^{-1}\cF_0\subset\cF_0$.
The following result gives sufficient conditions for hypotheses~\eqref{eq:Hinv} and~\eqref{eq:Hinv2} to hold.

\begin{prop} \label{prop:Hinv}
Let $p\ge1$ and let $v\in L^\infty(\Lambda,\R)$ with $\int_\Lambda v\,d\mu=0$.  
\begin{itemize}
\item[(a)]
Suppose that there exists $C>0$, $\eps>0$ such that
\[
\Big|\int_\Lambda v\,w\circ T^n\,d\mu\Big|\le C|w|_\infty \,n^{-(p+\eps)}
\]
for all $\cF_0$-measurable $w\in L^\infty(\Lambda,\R)$, $n\ge1$.
Then
$\sum_{n=1}^\infty |\E_0(v\circ T^{-n})|_p<\infty$.
\item[(b)]
Suppose that there exists $C>0$, $\eps>0$ such that
\[
\int_\Lambda \diam(v(T^n W^s))\,d\mu\le C|w|_\infty \,n^{-(p+\eps)}
\]
for all $n\ge1$.  Then
$\sum_{n=0}^\infty |\E_0(v\circ T^n)-v\circ T^n|_p<\infty$.
\end{itemize}
\end{prop}

\begin{proof}
The arguments are standard.  See for example~\cite[Theorem~3.1]{DMNapp}.
\end{proof}

Throughout the remainder of this section, $L^p$ is shorthand for $L^p((\Lambda,\mu),\R^d)$ unless stated otherwise.

\subsection{WIP in the invertible setting}
\label{sec:WIPinv}

Define $W_n\in D([0,\infty),\R^d)$ as in~\eqref{eq:W}.
Let $\nu$ be any probability measure on $\Lambda$ absolutely continuous with respect to $\mu$.
In this subsection, we prove:

\begin{thm} \label{thm:WIPinv}
Let $v\in L^\infty$ with $\int_\Lambda v\,d\mu=0$,
and suppose that~\eqref{eq:Hinv} holds.  Then conclusions~(a) and~(b) of
Theorem~\ref{thm:IWIP} hold, and 
$W_n\to_w W$ in $D([0,\infty),\R^d)$ as $n\to\infty$ on $(\Lambda,\nu)$, where $W$ is a $d$-dimensional Brownian motion with covariance $\Sigma$.
\end{thm}

For $-\infty<\ell\le k<\infty$, define
\[
\chi_\ell^k=\sum_{j=\ell}^k a_j, \qquad
a_j=\begin{cases} 
\E_0(v\circ T^j) & j\le -1  \\
\E_0(v\circ T^j)-v\circ T^j & j\ge0 
\end{cases}.
\]
Also define $\chi=\sum_{j=-\infty}^\infty a_j$.
It follows from our assumptions that $\chi^k_\ell\in L^\infty$ for all $\ell\le k$ and $\chi\in L^1$.  Moreover, $\chi^k_{-k}\to\chi$ in $L^1$ as $k\to\infty$.

\begin{prop} \label{prop:E1}
\begin{itemize}
\item[(a)]
$\E_{-1}(\chi_{-k}^{-\ell})=\chi_{-k-1}^{-\ell-1}\circ T$
for all $k\ge\ell>0$.
\item[(b)]
$\E\big(\chi_{\ell+1}^{k+1}(\chi_\ell^k\circ T-\chi_{\ell+1}^{k+1})\big)=0$
for all $k\ge\ell\ge0$.
\end{itemize}
\end{prop}

\begin{proof}
(a) Since $\E_{-1}\E_0=\E_{-1}$ and $\E_{-1}(g\circ T)=(\E_0\, g)\circ T$,
\[
\E_{-1}(\chi_{-k}^{-\ell})
 =\sum_{j=-k}^{-\ell}\E_{-1}(v\circ T^j)
 =\sum_{j=-k}^{-\ell}(\E_0(v\circ T^{j-1}))\circ T=\chi_{-k-1}^{-\ell-1}\circ T.
\]

\noindent(b) Note that
\[
\chi_\ell^k\circ T- \chi_{\ell+1}^{k+1}
 =
\sum_{j=\ell}^k\{ \E_0(v\circ T^j)\circ T- \E_0(v\circ T^{j+1})
\},
\]
so $\chi_\ell^k\circ T-\chi_{\ell+1}^{k+1}$ is $\cF_0$-measurable.
Also $\E_0\chi_{\ell+1}^{k+1}=0$.  Hence
\begin{align*}
\E\big(\chi_{\ell+1}^{k+1}(\chi_\ell^k\circ T-\chi_{\ell+1}^{k+1})\big)
& =\E\E_0\big(\chi_{\ell+1}^{k+1}(\chi_\ell^k\circ T-\chi_{\ell+1}^{k+1})\big)
\\ & =\E\big((\chi_\ell^k\circ T-\chi_{\ell+1}^{k+1})\E_0\chi_{\ell+1}^{k+1}\big)=0
\end{align*}
as required.
\end{proof}

Write
\begin{equation} \label{eq:martinv}
v=m^{(k)}+\chi^k_{-k}\circ T-\chi^k_{-k}+a_{-k}-a_{k+1},\,\;k\ge1 
\quad\text{and}\quad
v=m+\chi\circ T-\chi.
\end{equation}
It is immediate from the definitions that 
$m^{(k)}\in L^\infty$ for all $k$, that $m\in L^1$ and that $m^{(k)}\to m$ in $L^1$.  
Moreover, we have the following result corresponding to Lemma~\ref{lem:L2}:

\begin{lemma} \label{lem:L2inv}
$m\in L^2$ and $m^{(k)}\to m$ in $L^2$ as $k\to\infty$.
\end{lemma}

\begin{proof}
For $k\ge\ell\ge0$,
\begin{align*}
m^{(k)}-m^{(\ell)} & 
=(\chi_{-k}^k-\chi_{-\ell}^\ell)-(\chi_{-k}^k-\chi_{-\ell}^\ell)\circ T+(a_{k+1}-a_{-k})-(a_{\ell+1}-a_{-\ell})
\\ & =(\chi_{-k}^{-\ell-1}+\chi_{\ell+1}^k)-(\chi_{-k}^{-\ell-1}+\chi_{\ell+1}^k)\circ T+(a_{k+1}-a_{-k})-(a_{\ell+1}-a_{-\ell})
\\ & =(\chi_{-k+1}^{-\ell}+\chi_{\ell+2}^{k+1})-(\chi_{-k}^{-\ell-1}+\chi_{\ell+1}^k)\circ T.
\end{align*}
Hence
$|m^{(k)}-m^{(\ell)}|_2\le A+B$
where
\[
A=|\chi_{-k+1}^{-\ell}-\chi_{-k}^{-\ell-1}\circ T|_2, \qquad
B= |\chi_{\ell+2}^{k+1}-\chi_{\ell+1}^k\circ T|_2.
\]
Now,
\[
A^2 = \E\big((\chi_{-k+1}^{-\ell})^2-2\chi_{-k+1}^{-\ell} \, \chi_{-k}^{-\ell-1}\circ T+(\chi_{-k}^{-\ell-1})^2\big).
\]
By Proposition~\ref{prop:E1}(a),
\begin{align*}
\E(\chi_{-k+1}^{-\ell} \, \chi_{-k}^{-\ell-1}\circ T)
& =\E\E_{-1}(\chi_{-k+1}^{-\ell} \, \chi_{-k}^{-\ell-1}\circ T)
=\E\big(\chi_{-k}^{-\ell-1}\circ T \,\E_{-1}(\chi_{-k+1}^{-\ell})\big)
\\ & =\E\big((\chi_{-k}^{-\ell-1}\circ T)^2\big )
=\E\big((\chi_{-k}^{-\ell-1})^2\big ).
\end{align*}
Hence
\begin{align*}
A^2 & =\E\big((\chi_{-k+1}^{-\ell})^2-(\chi_{-k}^{-\ell-1})^2\big)
 =\E\big(
 (\chi_{-k+1}^{-\ell}-\chi_{-k}^{-\ell-1})
(\chi_{-k+1}^{-\ell}+\chi_{-k}^{-\ell-1})
\big)
 \\ & =\E\big( (a_{-\ell}-a_{-k}) (\chi_{-k+1}^{-\ell}+\chi_{-k}^{-\ell-1})\big)
 \le |a_{-\ell}-a_{-k}|_\infty |\chi_{-k+1}^{-\ell}+\chi_{-k}^{-\ell-1}|_1
\\ & \le 4|v|_\infty\sum_{j=-\infty}^{-\ell}|\E_0(v\circ T^j)|_1.
\end{align*}

Next, 
\[
B^2 =\E\big((\chi_{\ell+2}^{k+1})^2-2\chi_{\ell+2}^{k+1}\,\chi_{\ell+1}^k\circ T
+(\chi_{\ell+1}^k)^2\big).
\]
By Proposition~\ref{prop:E1}(b),
\[
\E(\chi_{\ell+2}^{k+1}\,\chi_{\ell+1}^k\circ T)
=\E\big(\chi_{\ell+2}^{k+1}\,(\chi_{\ell+1}^k\circ T-\chi_{\ell+2}^{k+1})
+(\chi_{\ell+2}^{k+1})^2\big)
=\E\big((\chi_{\ell+2}^{k+1})^2\big).
\]
Hence
\begin{align*}
B^2 & = \E\big((\chi_{\ell+1}^k)^2-(\chi_{\ell+2}^{k+1})^2\big)
=\E\big((\chi_{\ell+1}^k-\chi_{\ell+2}^{k+1})(\chi_{\ell+1}^k+\chi_{\ell+2}^{k+1})\big)
\\ & =\E\big((a_{\ell+1}-a_{k+1})(\chi_{\ell+1}^k+\chi_{\ell+2}^{k+1})\big)
\le |a_{\ell+1}-a_{k+1}|_\infty|\chi_{\ell+1}^k+\chi_{\ell+2}^{k+1}|_1
\\ & \le 8|v|_\infty \sum_{j=\ell+1}^\infty |\E_0(v\circ T^j)-v\circ T^j|_1.
\end{align*}

It follows from hypothesis~\eqref{eq:Hinv} together with these estimates for $A$ and $B$ that 
$m^{(k)}$ is Cauchy in $L^2$.
By uniqueness of limits in $L^1$, the $L^2$ limit of $m^{(k)}$ coincides with $m$.
\end{proof}

Standard calculations (see for example~\cite{Heyde75,Volny93} or~\cite[Proposition~2.2]{DMNapp}) show that $m$ is $\cF_0$-measurable and that $\E_{-1}m=0$.
Hence $\{m\circ T^{-j}:n\in\Z\}$ is a martingale with respect to the filtration $\cF_j$.  The same is true for 
\[
m^{(k)}=\sum_{j=-k+1}^{k+1}\E_0(v\circ T^j)-\sum_{j=-k}^k (\E_0(v\circ T^j))\circ T.
\]

\subsubsection*{Maximal inequality for $a_{-k}$}

\begin{prop} \label{prop:Rioinv}
Let $w\in L^\infty$ and suppose that $w$ is $\cF_0$-measurable.  
Then 
$\big|\max_{1\le\ell\le n}|w_\ell|\big|_2^2 \le 128\, n|w|_\infty\sum_{j=0}^\infty |\E_0(w\circ T^{-j})|_1$.
\end{prop}

\begin{proof}
Fix $n\ge1$ and
define the random variables $X(j)=w\circ T^{n-j}$ which are adapted to the filtration $\cF_{j-n}$.
Using Rio's inequality as in the proof of Proposition~\ref{prop:Rio},
$\big|\max_{1\le\ell\le n}|X(1)+\cdots+X(\ell)|\big|_2^2\le 16\sum_{j=1}^n b_{j,n}$
where
\[
b_{j,n} =\max_{1\le j\le u\le n}|X(j)\sum_{k=j}^u\E(X(k)|\cF_{j-n})|_1
\le |w|_\infty\max_{1\le j\le u\le n}\Big|\sum_{k=j}^u \E(w\circ T^{n-k}|\cF_{j-n})\Big|_1.
\]

Define $m_-$, $\chi_-\in L^1$, 
\[
\chi_-=\sum_{j=1}^\infty \E_0(w\circ T^{-j}), \qquad
w = m_-+\chi_-\circ T-\chi_-.
\]
Using that $w$ is $\cF_0$-measurable,
it is easily verified that $m_-$ is $\cF_0$-measurable and 
$\E_{-1}m_-=0$.  Hence 
$\E(m_-\circ T^{n-k}|\cF_{j-n})=0$ for all $k>j$.  
It follows that 
\[
\sum_{k=j}^u \E(w\circ T^{n-k}|\cF_{j-n})=\E\big(m_-\circ T^{n-j}+
\chi_-\circ T^{n+1-j} -\chi_-\circ T^{n-u}|\cF_{j-n}\big).
\]
Now continue as in the proof of Proposition~\ref{prop:Rio}.
\end{proof}

\begin{cor} \label{cor:Rioinv}
$\big|\max_{1\le\ell\le n}|(a_{-k})_\ell|\big|_2^2 
\le 128\, n|v|_\infty\sum_{j=k}^\infty |\E_0(v\circ T^{-j})|_1$.
\end{cor}

\begin{proof}
Recall that $a_{-k}=\E_0(v\circ T^{-k})$, so $|a_{-k}|_\infty\le |v|_\infty$
and $a_{-k}$ is $\cF_0$-measurable.  
By Proposition~\ref{prop:Rioinv},
\[
\big|\max_{1\le\ell\le n}|(a_{-k})_\ell|\big|_2^2 \le 128\, n|v|_\infty\sum_{j=0}^\infty |\E_0(a_{-k}\circ T^{-j})|_1.
\]
Setting $g=v\circ T^{-k}$,
\begin{align*}
\E_0(a_{-k}\circ T^{-j}) & = \E_0((\E_0g)\circ T^{-j}) 
=(\E_{-j}\E_0g)\circ T^{-j}
\\ & =(\E_{-j}g)\circ T^{-j} =\E_0(g\circ T^{-j}) =\E_0(v\circ T^{-(j+k)}).
\end{align*}
The result follows.
\end{proof}

\subsubsection*{Maximal inequality for $a_k$}

Here we rely heavily on ideas from~\cite{DedeckerRio00}.
In particular, we require the following maximal inequality~\cite[Equation~(3.4)]{DedeckerRio00}:
\begin{lemma}  \label{lem:DR}
Let $S_n=\sum_{j=1}^n X(j)$ be a sum of $L^2$ random variables.
Then
\[
\E({S_n^*}^2)\le 4\E(S_n^{\,2})-4\sum_{j=1}^n\E(X(j)S_{j-1}^{\,^*})
\]
where $S_n^*=\max\{0,S_1,\dots,S_n\}$.  \qed
\end{lemma}

The following elementary estimate is useful:
\begin{prop} \label{prop:hLip}
Define $h:\R^n\to\R$, $h(b)=\max\{0,b_1,b_1+b_2,\dots,\sum_{j=1}^n b_i\}$.
Then $|h(b)-h(b')|\le \sum_{i=1}^n|b_i-b_i'|$.
\qed
\end{prop}

\begin{prop} \label{prop:DR}
Let $w\in L^\infty$ with $\E_0w=0$.  Then 
\[
\Big|\max_{1\le\ell\le n}|w_\ell|\big|_2^2
\le 96\,n|w|_\infty\sum_{j=0}^\infty|\E_0(w\circ T^j)-w\circ T^j|_1.
\]
\end{prop}

\begin{proof}
Define $ X(j)=w\circ T^{-j}$ and $S_n=\sum_{j=1}^n X(j)$.
Then
\[
\E(S_n^2) =\sum_{i,j=0}^{n-1}\E(w\circ T^{-i}\,w\circ T^{-j})
=n\E(w^2)+2\sum_{j=1}^{n-1}(n-j)\E(w\,w\circ T^j).
\]
Also,
\(
\E\big(w\,\E_0(w\circ T^j)\big)=
\E\big(\E_0(w\circ T^j)\,\E_0w\big)=0
\)
and so
\begin{align*}
\E(S_n^2)
& = n\E(w^2)+2\sum_{j=1}^{n-1}(n-j)\E\big(w\,(w\circ T^j-\E_0(w\circ T^j))\big)
\\ & \le 2n|w|_\infty\Big(|w|_1+\sum_{j=1}^\infty|\E_0(w\circ T^j)-w\circ T^j|_1\Big)
\\ & = 2n|w|_\infty\sum_{j=0}^\infty|\E_0(w\circ T^j)-w\circ T^j|_1.
\end{align*}

Next, define 
\[
Y_{i,j}=\E_j(w\circ T^{-i}),
\qquad Z_{p,j}=\sum_{i=1}^p Y_{i,j},
\qquad  Z_{j-1}^*=\max\{0,Z_{1,j},\dots,Z_{j-1,j}\}.
\]
Note that $Y_{i,j}$ is $\cF_j$-measurable for all $i<j$, so 
in particular $Z_{j-1}^*$ is $\cF_j$-measurable.
Hence $\E(X(j)Z_{j-1}^*)=\E(Z_{j-1}^*\E_jX(j))=0$.
It follows that
\[
\sum_{j=1}^n|\E(X(j)S_{j-1}^*)|=
\sum_{j=1}^n|\E(X(j)(S_{j-1}^*-Z_{j-1}^*))|
\le |w|_\infty \sum_{j=1}^n\E|S_{j-1}^*-Z_{j-1}^*|.
\]
By Proposition~\ref{prop:hLip},
\[
|S_{j-1}^*-Z_{j-1}^*|\le \sum_{i=1}^{j-1}|X(i)-Y_{i,j}|=
\sum_{i=1}^{j-1}|w\circ T^{-i}-(\E_{j-i}w)\circ T^{-i}|
\]
and hence
\begin{align*}
\sum_{j=1}^n|\E(X(j)S_{j-1}^*)| & 
\le |w|_\infty \sum_{1\le i<j\le n}|w-\E_{j-i}w|_1
= |w|_\infty \sum_{j=1}^{n-1}(n-j)|w-\E_jw|_1
\\ & \le n|w|_\infty\sum_{j=1}^\infty|\E_jw-w|_1
= n|w|_\infty\sum_{j=1}^\infty|(\E_0(w\circ T^j))\circ T^{-j}-w|_1
\\ & = n|w|_\infty\sum_{j=1}^\infty|\E_0(w\circ T^j)-w\circ T^j|_1.
\end{align*}
Combining this with the estimate for $\E S_n^2$ it follows from
Lemma~\ref{lem:DR} that
\[
\E({S_n^*}^2)\le 12n|w|_\infty
\sum_{j=0}^\infty|\E_0(w\circ T^j)-w\circ T^j|_1.
\]

The transformation $w\mapsto-w$ sends $S_n^*\mapsto 
S_{n,*}=\max\{0,-S_1,\dots,-S_n\}$.
Hence $\E(S_{n,*}^2)\le 12n|w|_\infty
\sum_{j=0}^\infty|\E_0(w\circ T^j)-w\circ T^j|_1$, and so
\[
\max_{1\le\ell\le n}|S_\ell|^2=
\max\{{S_n^*}^2,S_{n,*}^2\}
\le {S_n^*}^2+S_{n,*}^2
\le 24\,n|w|_\infty
\sum_{j=0}^\infty|\E_0(w\circ T^j)-w\circ T^j|_1.
\]

Finally,
$w_\ell=(S_n-S_{n-\ell})\circ T^n$,
so $\big|\max_{1\le\ell\le n}|w_\ell|\big|_2\le 2\big|\max_{1\le\ell\le n}|S_\ell|\big|_2$ and the result follows.
\end{proof}

\begin{cor} \label{cor:DR}
\(
\big|\max_{1\le\ell\le n}|(a_k)_\ell|\big|_2^2 
\le 192\, n|v|_\infty\sum_{j=k}^\infty |\E_0(v\circ T^j)-v\circ T^j|_1.
\)
\end{cor}

\begin{proof}
Recall that $a_k=\E_0(v\circ T^k)-v\circ T^k$, so $|a_k|_\infty\le 2|v|_\infty$
and $\E_0a_k=0$.
By Proposition~\ref{prop:DR},
\[
\big|\max_{1\le\ell\le n}|(a_k)_\ell|\big|_2^2 \le 192\, n|v|_\infty\sum_{j=0}^\infty |\E_0(a_k\circ T^j)-a_k\circ T^j|_1.
\]
Setting $g=v\circ T^k$, $a_k=\E_0g-g$,
\begin{align*}
\E_0(a_k\circ T^j)-a_k\circ T^j
& =\E_0((\E_0g)\circ T^j)-\E_0(g\circ T^j)-(\E_0g)\circ T^j+g\circ T^j
\\ & =(\E_j\E_0g)\circ T^j-\E_0(g\circ T^j)-(\E_0g)\circ T^j+g\circ T^j
\\ & =-\E_0(g\circ T^j)+g\circ T^j
 =-\E_0(v\circ T^{j+k})+v\circ T^{j+k}.
\end{align*}
The result follows.
\end{proof}

\begin{lemma} \label{lem:diffinv}
$\lim_{n\to\infty}\frac{1}{\sqrt n} \big|\max_{1\le\ell\le n}|(v-m)_\ell|\big|_2=0$.
\end{lemma}

\begin{proof}
By Lemma~\ref{lem:L2inv} and hypothesis~\eqref{eq:Hinv}, for each $\eps>0$, there exists $k\ge1$ such that 
\[
|m-m^{(k)}|_2<\eps, \quad
\sum_{j=k}^\infty |\E_0(v\circ T^{-j})|_1<\eps^2, \quad
\sum_{j=k+1}^\infty |\E_0(v\circ T^j)-v\circ T^j|_1<\eps^2.
\]

Since $\E_{-1}m=\E_{-1}m^{(k)}=0$, it follows from Doob's inequality as in
Corollary~\ref{cor:Doob} that
\begin{equation} \label{eq:diffinv0}
\big|\max_{1\le\ell\le n}|(m-m^{(k)})_\ell|\big|_2<4\sqrt n\,\eps.
\end{equation}

By~\eqref{eq:martinv},
\begin{align} \label{eq:diffinv1}
\nonumber
|(v-m^{(k)})_n| & \le 
2|\chi_{-k}^k|_\infty+|(a_{-k})_n|+|(a_{k+1})_n|
 \\ & \le (6k+2)|v|_\infty+|(a_{-k})_n|+|(a_{k+1})_n|.
\end{align}

Substituting the estimates from Corollaries~\ref{cor:Rioinv} and~\ref{cor:DR} into~\eqref{eq:diffinv1},
\(
\big|\max_{1\le\ell\le n}|(v-m^{(k)})_\ell|\big|_2  \ll
 k+\eps n^{1/2},
\)
and combining this with~\eqref{eq:diffinv0},
\[
\frac{1}{\sqrt n}\big|\max_{1\le\ell\le n}|(v-m)_\ell|\big|_2  \ll kn^{-1/2}+\eps.
\]
Hence $\limsup_{n\to\infty} \frac{1}{\sqrt n}\big|\max_{1\le\ell\le n}|(v-m)_\ell|\big|_2\ll \eps$ and the result follows
since $\eps$ is arbitrary.
\end{proof}

We require the following standard result from probability theory.

\begin{prop} \label{prop:max}
Let $Y_1,Y_2.\ldots$ be identically distributed random variables with finite second moment.
Then $|\max_{1\le\ell\le n}\big|Y_\ell|\big|_2=o(\sqrt n)$
as $n\to\infty$. \qed
\end{prop}

\begin{pfof}{Theorem~\ref{thm:WIPinv}}
Conclusions~(a) and~(b) hold by the same arguments in the proof of Theorem~\ref{thm:IWIP} (using Lemma~\ref{lem:diffinv} 
in place of Lemma~\ref{lem:diff}).

Fix $K\ge1$ to be an integer
and define $W_n^-(t)$ and $M_n^-(t)$ for $t\in[0,K]$ as in~\eqref{eq:WM-}.  Also, for $t\ge0$ define 
\[
\tM_n^-(t)=\frac{1}{\sqrt n}\sum_{j=1}^{[nt]}m\circ T^{-j}.
\]
Note that $\{m\circ T^{-n};\,n\in\Z\}$ is a martingale difference sequence with respect to the filtration $\cF_n$.
By Proposition~\ref{prop:max},
$\frac{1}{\sqrt n}\big|\max_{1\le j\le n}|m\circ T^{-j}|\big|_2\to0$.
Also, by the ergodic theorem
$\frac1n\sum_{j=1}^n(m\otimes m)\circ T^{-j}\to\int_\Lambda m\otimes m\,d\mu=\Sigma$ almost everywhere.
Hence we have verified the hypotheses of~\cite[Theorem~2.1]{Whitt07},
yielding $\tM_n^-\to_w W$ in $D[0,\infty),\R^d)$.
Since $M_n^-=\tM_n^-\circ T^{nK}$, 
it follows that $M_n^-\to_w W$ in $D[0,K],\R^d)$.
By~\eqref{eq:diff-} and Lemma~\ref{lem:diffinv},
$W_n^-\to_w W$ in $D[0,K],\R^d)$ on $(\Lambda,\mu)$.

Defining $g$ as in Step~1 of the proof of Lemma~\ref{lem:BBM},
we obtain 
\[
W_n(t)=g(W_n^-(t))-F_n^1(t) \quad\text{for $t\in[0,K]$,}
\]
where $\sup_{t\in[0,K]}|F_n^1(t)|\le n^{-1/2}|v|_\infty$.
Applying~\cite[Proposition~4.9 and Lemma~4.11]{KM16}, $W_n\to_w g(W)=_w W$ in $D[0,K],\R^d)$ on $(\Lambda,\mu)$.

Finally, the case where $\nu$ is a general probability measure absolutely continuous with respect to $\mu$ follows from~\cite[Corollary~3]{Zweimuller07}.
\end{pfof}

\subsection{Iterated WIP in the invertible setting}
\label{sec:IWIPinv}

Define 
$W_n\in D([0,\infty),\R^d)$, 
$\BBW_n\in D([0,\infty),\R^{d\times d})$ as in~\eqref{eq:W}.
Let $\nu$ be any probability measure on $\Lambda$ absolutely continuous with respect to $\mu$.
In this subsection, we prove:

\begin{thm} \label{thm:IWIPinv}
Let $v\in L^\infty$ with $\int_\Lambda v\,d\mu=0$,
and suppose that~\eqref{eq:Hinv2} holds.  Assume also
that $T$ is mixing.
Then 
$(W_n,\BBW_n)\to_w (W,\BBW)$ in $D([0,\infty),\R^d\times\R^{d\times d})$ as $n\to\infty$ on $(\Lambda,\nu)$, where 
$W$ is as in Theorem~\ref{thm:WIPinv} and
\[
\BBW(t)=\int_0^t W \otimes dW+t\sum_{j=1}^\infty \int_\Lambda v\otimes(v\circ T^j)\,d\mu.
\]
\end{thm}

Write  
\[
\chi=\chi_-+\chi_+,
\qquad
\chi_-=\sum_{j=1}^\infty \E_0(v\circ T^{-j}),
\qquad
\chi_+=\sum_{j=0}^\infty (\E_0(v\circ T^j)-v\circ T^j).
\]
By~\eqref{eq:Hinv2}, $\chi_-\in L^1$ and $\chi_+\in L^2$.
Define $\hv\in L^2$ by
\[
v=\hv+\chi_+\circ T-\chi_+.
\]
Then 
\[
\hv=\sum_{j=0}^\infty\big\{ \E_0(v\circ T^j)-\big(\E_0(v\circ T^j)\big)\circ T\big\}
\]
 is $\cF_0$-measurable.

Define 
\[
\hBBM_n\in D([0,\infty),\R^{d\times d}), \qquad
\hBBM_n(t)=\frac1n\sum_{0\le i< j\le[nt]-1}(m\circ T^i)\otimes (\hv\circ T^j),
\]
where $m$ is as in~\eqref{eq:martinv}.
(This differs from the definition of $\BBM_n$ in Section~\ref{sec:Non}; we use $\hat v$ instead of $v$ since $v$ is not $\cF_0$-measurable.)

\begin{lemma} \label{lem:BBMinv}
$(W_n,\hBBM_n)\to_w (W,\BBM)$ in $D([0,\infty),\R^d\times\R^{d\times d})$ as $n\to\infty$ on $(\Lambda,\mu)$, where $\BBM(t)=\int_0^t W\otimes dW$.
\end{lemma}

\begin{proof}
Fix $K\ge1$ an integer and define $W_n^-$, $M_n^-$ as in~\eqref{eq:WM-}.
As shown in the proof of Theorem~\ref{thm:WIPinv},
$M_n^-\to_w W$ in $D([0,[0,K],\R^d)$.
By the continuous mapping theorem,
$(M_n^-,M_n^-)\to_w (W,W)$ in $D([0,K]),\R^d\times\R^d)$.

Define 
\[
\hW_n(t)=\frac1n\sum_{0\le j\le [nt]-1} \hv\circ T^j, \qquad
\hW_n^-(t)=\frac1n\sum_{1\le j\le [nt]} \hv\circ T^{nK-j}.
\]
By Lemma~\ref{lem:diffinv},
$\frac{1}{\sqrt n}\big|\max_{1\le \ell\le n}|(v-m)_\ell|\big|_2\to0$.
Also,
$\frac{1}{\sqrt n}\big|\max_{1\le \ell\le n}|(v-\hat v)_\ell|\big|_2
\le \frac{2}{\sqrt n}|\max_{1\le \ell\le n}\chi_+\circ T^\ell|_2\to0$ 
by Proposition~\ref{prop:max}.
Hence $\frac{1}{\sqrt n}\big|\max_{1\le \ell\le n}|(\hv-m)_\ell|\big|_2\to0$.
It follows that
\[
(\hW_n^-,M_n^-)\to_w (W,W)
\quad \text{ in $D([0,K],\R^d\times\R^d)$.}
\]

Define 
\[
\hBBM_n^-(t)=\frac1n\sum_{1\le j< i\le [nt]}(\hv\circ T^{-j})\otimes (m\circ T^{-i}).
\]
We apply~\cite{JakubowskiMeminPages89,KurtzProtter91} as in Step~2 of the proof of Lemma~\ref{lem:BBM}: $M_n^-$ is a martingale and $\hW_n^-$ is adapted with respect to the filtration $\cF_j$.  Moreover, 
$\int_\Lambda|M_n^-(t)|^2\,d\mu=
n^{-1}[nt]\int_\Lambda|m|^2\,d\mu\le K|m|_2^2$ for all $t\in[0,K]$
so 
condition~C2.2(i) in~\cite[Theorem~2.2]{KurtzProtter91} is satisfied.
Hence
\[
(\hW_n^-, M_n^-, \hBBM_n^-) \to_w (W,W,\BBM)
\quad\text{in $D([0,K],\R^d\times\R^d\times\R^{d\times d})$.}
\]

Next, define 
$h:D([0,K],\R^d\times\R^d\times\R^{d\times d})\to
\tD([0,K],\R^d\times\R^{d\times d})$,
as in Step~3 of the proof of Lemma~\ref{lem:BBM}.
Then 
\[
\SMALL (\hW_n,\hBBM_n)=h(\hW_n^-,M_n^-,\hBBM_n^-)-F_n\quad\text{where}\quad
\sup_{t\in[0,K]}|F_n(t)|\to_p 0,
\]
and we deduce that 
\[
(\hW_n, \hBBM_n) \to_w (W,\BBM)
\quad\text{in $D([0,K],\R^d\times\R^{d\times d})$.}
\]
Using once again that
$\frac{1}{\sqrt n}\big|\max_{1\le \ell\le n}|(v-\hat v)_\ell|\big|_2 \to0$,
we obtain the desired result.
\end{proof}

\begin{prop} \label{prop:drift}
$\sum_{j=1}^\infty \int_\Lambda v\otimes(v\circ T^j)\,d\mu
=\int_\Lambda \big(\chi\otimes v-m\otimes(\chi_+\circ T) \big)\,d\mu$.
\end{prop}

\begin{proof}
Write
\begin{align*}
v\otimes(v\circ T^j)
& =(m+\chi\circ T-\chi)\otimes (v\circ T^j) \\
& = 
m\otimes(\hv\circ T^j+\chi_+\circ T^{j+1}-\chi_+\circ T^j)
+(\chi\circ T-\chi)\otimes (v\circ T^j).
\end{align*}
Then
\(
\sum_{j=1}^n \int_\Lambda v\otimes(v\circ T^j)\,d\mu=I_1+I_2+I_3
\)
where
\begin{align*}
I_1 & =
\sum_{j=1}^n\int_\Lambda m\otimes(\hv\circ T^j)\,d\mu, \qquad
 I_2  = \int_\Lambda m\otimes(\chi_+\circ T^{n+1}-\chi_+\circ T)\,d\mu,
\\ I_3 & = 
\int_\Lambda\sum_{j=1}^n (\chi\circ T-\chi)\otimes v\circ T^j)\,d\mu. 
\end{align*}
Now, 
\[
\E\big(m\otimes(\hv\circ T^j)\big)
=\E\E_{-j}\big(m\otimes(\hv\circ T^j)\big)
=\E\big((\E_{-j}m)\otimes(\hv\circ T^j)\big)=0,
\]
so $I_1=0$.  Since $T$ is mixing, $I_2\to -\int_\Lambda m\otimes(\chi_+\circ T)\,d\mu$.
Finally,
\[
I_3=\int_\Lambda\sum_{j=1}^n (\chi\circ T^{-(j-1)}-\chi\circ T^{-j})\otimes v\,d\mu
=\int_\Lambda(\chi-\chi\circ T^{-n})\otimes v\,d\mu
\to \int_\Lambda \chi\otimes v\,d\mu
\]
since $T$ is mixing.
\end{proof}

\begin{pfof}{Theorem~\ref{thm:IWIPinv}}
Write
\begin{align*}
(v\circ T^i)\otimes(v\circ T^j) 
& = (m\circ T^i)\otimes(v\circ T^j) +
(\chi\circ T^{i+1}-\chi\circ T^i)\otimes(v\circ T^j)  \\
& = (m\circ T^i)\otimes(\hv\circ T^j) +
 (m\circ T^i)\otimes(\chi_+\circ T^{j+1}-\chi_+\circ T^j) 
\\ & \qquad\qquad \qquad\qquad
 \qquad\qquad +
(\chi\circ T^{i+1}-\chi\circ T^i)\otimes(v\circ T^j).
\end{align*}
Then
\begin{align*}
\BBW_n(t)-\hBBM_n(t)
 & = 
\frac1n \sum_{0\le i\le [nt]-2}(m\circ T^i)\otimes (\chi_+\circ T^{[nt]}-\chi_+\circ T^{i+1})
\\ & \qquad\qquad +\frac1n \sum_{1\le j\le [nt]-1}(\chi\circ T^j-\chi)\otimes (v\circ T^j)=\frac1n(A_n(t)+B_n(t))
\end{align*}
where
\begin{align*}
A_n(t) & = 
\sum_{0\le i\le[nt]-2}(m\circ T^i)\otimes (\chi_+\circ T^{[nt]})
\\
B_n(t) & = 
-\sum_{0\le i\le[nt]-2}(m\otimes (\chi_+\circ T))\circ  T^i
+ \sum_{1\le j\le[nt]-1}(\chi\otimes v)\circ T^j
- \chi\otimes\sum_{1\le j\le[nt]-1}v\circ T^j.
\end{align*}

Recall that $v\in L^\infty$, $\chi\in L^1$ and $\chi_+,\,m\in L^2$.  
By the ergodic theorem, 
$\frac1n B_n(1)\to
\int_\Lambda ( \chi\otimes v - m\otimes(\chi_+\circ T))\,d\mu \;\text{a.e.}$
and hence
\[
\frac1n\sup_{t\in[0,K]}\Big|B_n(t)-t\int_\Lambda ( \chi\otimes v - m\otimes(\chi_+\circ T)) \,d\mu\Big|\to0
\quad\text{a.e.}
\]
Also,
\begin{align*}
\frac1n\big|\sup_{t\in[0,K]}A_n(t)\big|_1
& \le 
\frac{1}{\sqrt n}\Big|\max_{1\le \ell\le nK}\big|\sum_{i<\ell}m\circ T^i\big|\Big|_2
\,\frac{1}{\sqrt n}\big|\max_{1\le \ell\le nK}\chi_+\circ T^\ell\big|_2\to0
\\ & \le 2\sqrt{K}|m|_2  
\frac{1}{\sqrt n}\big|\max_{1\le \ell\le nK}\chi_+\circ T^\ell\big|_2\to0
\end{align*}
by Doob's inequality and Proposition~\ref{prop:max}.
Hence
\[
\sup_{t\in[0,K]}\Big|\BBW_n(t)-\hBBM_n(t)-
t\int_\Lambda \big(\chi\otimes v-m\otimes(\chi_+\circ T) \big)\,d\mu\Big|\to_p0.
\]
By this 
combined with Lemma~\ref{lem:BBMinv} and Proposition~\ref{prop:drift},
$(W_n,\BBW_n)\to_w (W,\BBW)$ on $(\Lambda,\mu)$.

Finally, we consider the case where $\nu$ is a general probability measure absolutely continuous with respect to $\mu$.
As in the proof of Theorem~\ref{thm:IWIP}, it suffices to establish~\eqref{eq:bfW}
for all $\eps>0$.
The estimates
\[
|W_n(t)\circ T-W_n(t)|\le 2n^{-1/2}|v|_\infty,\qquad 
|\BBW_n(t)\circ T-\BBW_n(t)|\le 
2n^{-1}|v|_\infty \max_{1\le k\le nK}|v_k|
\]
hold as before
for all $t\in[0,K]$.
Hence it suffices to show that $\big|\max_{1\le k\le nK}|v_k|\big|_2\ll n^{1/2}$.

Write $v=w+w'$ where $w=\E_0v$, $w'=v-\E_0v$.
By Propositions~\ref{prop:Rioinv}  and~\ref{prop:DR},
\[
\big|\max_{1\le k\le nK}|w_k|\big|_2\ll 
\Big(n|v|_\infty \sum_{j\ge0}|\E_0(w\circ T^{-j})|_1\Big)^{1/2} 
=\Big(n|v|_\infty \sum_{j\ge0}|\E_0(v\circ T^{-j})|_1\Big)^{1/2} 
\]
and 
\begin{align*}
\big|\max_{1\le k\le nK}|w'_k|\big|_2 & \ll 
\Big(n|v|_\infty \sum_{j\ge0}|\E_0(w'\circ T^j)-w'\circ T^j|_1\Big)^{1/2} 
\\ & =\Big(n|v|_\infty \sum_{j\ge0}|\E_0(v\circ T^j)-v\circ T^j|_1\Big)^{1/2}.
\end{align*}
Hence the required estimate for $\max_{1\le k\le nK}|v_k|$ follows from~\eqref{eq:Hinv}.
\end{pfof}

\section{Examples}
\label{sec:ex}

In this section, we consider examples consisting of time-one maps of nonuniformly expanding semiflows and
nonuniformly hyperbolic flows to which our theory applies and gives new results.

\subsection{Noninvertible setting}
\label{sec:ex1}

We begin by revisiting nonuniformly expanding maps modelled by one-sided Young towers~\cite{Young99}.
Optimal results for the iterated WIP were obtained by~\cite{KM16} and we recover their result.  
In particular,~\cite{Young99} proved results on decay of correlations; Theorem~\ref{thm:IWIP} applies whenever the decay of correlations is summable by Proposition~\ref{prop:H}.
As described below, we are moreover able to treat time-one maps of nonuniformly expanding semiflows, significantly improving on existing results.

It is convenient to mention a specific family of dynamical systems.
Prototypical examples of nonuniformly expanding map are given by intermittent maps of Pomeau-Manneville type~\cite{PomeauManneville80}.
For definiteness, we consider the example considered by~\cite{LSV99}, namely
\begin{align} \label{eq:LSV}
f:[0,1]\to[0,1], \qquad f(x)=\begin{cases} x(1+2^\gamma x^\gamma) & x<\frac12 \\
2x-1 & x>\frac12 \end{cases}.
\end{align}
Here $\gamma\in[0,1)$ is a parameter and there is a unique absolutely continuous invariant probability measure $\mu_0$ for each $\gamma$.

Let $v:[0,1]\to\R^d$ be H\"older with $\int_0^1 v\,d\mu_0=0$.
By~\cite{Hu04,Young99}, there is a constant $C>0$ such that $|\int_0^1 v\,w\circ T^m\,d\mu_0|\le Cn^{-(\gamma^{-1}-1)}|w|_\infty$ for all $w\in L^\infty([0,1],\R)$.  By Proposition~\ref{prop:H},
hypothesis~\eqref{eq:H} holds for $\gamma<\frac12$.
Hence we obtain the iterated WIP, Theorem~\ref{thm:IWIP}, 
for all $\gamma<\frac12$.
This recovers a result of~\cite[Example~10.3]{KM16} and it is sharp since even the CLT fails for $\gamma\in[\frac12,1)$ when $v(0)\neq0$ by~\cite{Gouezel04}.

Now we consider suspension semiflows and their time-one maps to obtain new examples where the iterated  WIP holds.  Again we consider the specific example~\eqref{eq:LSV} for definiteness, but $f$ could be replaced by any nonuniformly expanding map modelled by a Young tower.  Let $h:[0,1]\to(0,\infty)$ be a H\"older roof function and define the suspension semiflow
$f_t:\Lambda\to\Lambda$ where
\[
\Lambda=\{(x,u)\in[0,1]\times\R:0\le u\le h(x)\}/\sim,
\qquad (x,h(x))\sim(fx,0),
\]
and $f_t(x,u)=(x,u+t)$ computed modulo identifications.
The probability measure $\mu=(\mu_0\times\Leb)/\int_0^1 h\,d\mu_0$
is $f_t$-invariant and ergodic.  
At the level of the semiflow $f_t$, when $\gamma<\frac12$
the iterated WIP holds for H\"older mean zero observables $v$ by~\cite[Theorem~6.1]{KM16}.

Now consider the time-one map $T=f_1:\Lambda\to\Lambda$.
In general, even the CLT is not known for such maps.  By~\cite{M09,M18},
typically (under a non-approximate eigenfunction condition due to~\cite{Dolgopyat98b}) $T$ has decay of correlations at the same rate as $f$ for sufficiently smooth observables.  
(We refer to~\cite[Section~3]{M18} for details regarding the class of observables $v$ and~\cite[Section~5]{M18} for details regarding the word ``typical''.)
A consequence~\cite[Theorem~4.3 and Proposition~4.4]{KM16} is that the iterated WIP holds for $\gamma<\frac13$.
Previously, the range $\gamma\in[\frac13,\frac12)$ remained open.
But Proposition~\ref{prop:H} again implies that the $L^1$ Gordin criterion~\eqref{eq:H} holds for all $\gamma<\frac12$.
Hence, by Theorem~\ref{thm:IWIP}, the iterated WIP holds in the optimal range $\gamma<\frac12$.

\subsection{Invertible setting}
We begin by revisiting nonuniformly hyperbolic maps modelled by two-sided Young towers~\cite{Young98,Young99}.
Optimal results in this setting were obtained 
by~\cite{MV16} (see also~\cite[Section~10.2]{KM16}).
(Unlike in the noninvertible setting, Theorem~\ref{thm:IWIPinv} does not recover this result since the iterated WIP is not known under the $L^1$ Gordin criterion~\eqref{eq:Hinv}.)

Examples of nonuniformly hyperbolic maps include intermittent solenoids~\cite[Section~5]{AlvesPinheiro08} and~\cite[Example 4.2]{MV16}.  These are invertible analogues of the intermittent maps in Subsection~\ref{sec:ex1} and are obtained by adapting the classical Smale-Williams solenoid construction~\cite{Smale67,Williams67}.  There is an invariant  contracting stable foliation $\cW^s$ as in Proposition~\ref{prop:Hinv} and the dynamics modulo the stable leaves is given by an intermittent map.  In particular, condition~(a) in Proposition~\ref{prop:Hinv} is satisfied for $p<\gamma^{-1}-1$ where $\gamma$ is the parameter for the intermittent map.
The examples in~\cite{AlvesPinheiro08} and some of the examples in~\cite{MV16} have exponential contraction along
stable leaves.  For these examples and $v$ H\"older, condition~(b) in Proposition~\ref{prop:Hinv} is satisfied for all $p\in[1,\infty)$ and hence Theorem~\ref{thm:IWIPinv} applies for all $\gamma<\frac12$.  The remaining examples in~\cite{MV16} have contraction along stable leaves which is as slow as the expansion of the underlying intermittent map, and condition~(b) in Proposition~\ref{prop:Hinv} is satisfied for $p<\gamma^{-1}-1$; hence Theorem~\ref{thm:IWIPinv} applies for $\gamma<\frac13$.  

As in Subsection~\ref{sec:ex1}, we consider intermittent solenoidal flows given by suspensions over intermittent solenoids.  Optimal results on the iterated WIP for such flows follow by combining~\cite{MV16} and~\cite[Theorem~6.1]{KM16}.  
Again we focus on time-one maps of intermittent solenoid flows, restricting to typical flows and sufficiently smooth observables.  Previous results on the iterated WIP in this context apply only for $\gamma<\frac13$;  
we considerably relax this restriction.
By the arguments in~\cite{DMNapp}, 
the conditions of Proposition~\ref{prop:Hinv} hold for the same values of $p$ as in the case of intermittent solenoids described above.  Hence for the examples with exponential contraction along stable leaves, we obtain the iterated WIP for all $\gamma<\frac12$.  Exponential contraction can be relaxed to moderately fast polynomial contraction as discussed in~\cite{AlvesAzevedo16}.
Explicit examples with $\gamma\in[\frac13,\frac12)$ and condition~(b) of Proposition~\ref{prop:Hinv} holding for $p=2$ can be found in~\cite{GaltonPhD}.
For such examples, the iterated WIP follows from Theorem~\ref{thm:IWIPinv}; this is far beyond the scope of previous methods.

\section{Application to homogenisation}
\label{sec:fs}

Let $(\Lambda,\cF,\mu)$ be a probability space and 
$T:\Lambda\to\Lambda$ be an ergodic measure-preserving map.
Consider the fast-slow system 
\begin{align*}
\x_{n+1} & = \x_n+ \eps^2 a(\x_n)+\eps b(\x_n)v(y_n), \quad \x_0=\xi\in\R^d,  \\
y_{n+1}  & = Ty_n,
\end{align*}
where $a:\R^d\to\R^d$,  $b:\R^d\to\R^{d\times d}$ and
$v\in L^\infty(\Lambda,\R^d)$ with $\int_\Lambda v\,d\mu=0$.

Define $\hx(t)=\x_{[t/\eps^2]}$ and
\[
W_\eps(t)=\eps\sum_{0\le j\le [t/\eps^2]-1}v(y_j), \qquad
\BBW_\eps(t)=\eps^2\sum_{0\le i<j\le [t/\eps^2]-1} v(y_i)\otimes v(y_j).
\]

The aim is to prove homogenisation to a stochastic differential equation (SDE)
of the type 
\[
dX=\tilde a(X)\,dt + b(X)\, dW, \quad X(0)=\xi,
\]
where $\tilde a:\R^d\to\R^d$ is to be determined; i.e.\ to show that
$\hx\to_w X$ in $D[0,\infty),\R^d)$ as $\eps\to0$.

In the special case where $a$ and $b$ are Lipschitz, and $b$ satisfies an exactness condition of the form
$b=(dh)^{-1}$ for some $h:\R^d\to\R^d$, this problem was completely solved by~\cite{GM13b}: It is necessary and sufficient that $v$ satisfies the WIP.  Hence the $L^1$ Gordin criterion suffices by Theorems~\ref{thm:IWIP} and~\ref{thm:WIPinv}.

When the exactness condition for  $b$ fails,~\cite{KM16,KKMsub} proved homogenisation for
$a\in C^{1+}$ and $b\in C^{2+}$ under an
$L^4$ Gordin criterion on $v$ (see~\cite[Theorem~2.9]{DMNapp}).
Moreover, $\tilde a(X)=a(X)+\frac12\sum_{\alpha,\beta,\gamma=1}^d E^{\gamma\beta}\partial_\alpha b^\beta(X) b^{\alpha\gamma}(X)$
where $E$ is the matrix in the iterated WIP.\footnote{There is a typo in~\cite{KM16} and subsequent papers; the matrix entry $E^{\beta\gamma}$ should be replaced by
$E^{\gamma\beta}$ as written here.}

In certain special cases, our results yield homogenisation theorems where the previous papers do not.  One such example is the following:

\begin{prop}
Let $d=2$ and write $x=(x^1,x^2)$.  Let
\[
a(x)=\left(\begin{array}{c} 0   \\  g(x) \end{array}\right),
\qquad
b(x)=\left(\begin{array}{cc} 1  & 0 \\ 0 & x^1 \end{array}\right),
\]
where $g:\R^2\to\R$ is Lipschitz.
Suppose that $v$ satisfies either the $L^1$ Gordin criterion~\eqref{eq:H} in the noninvertible setting or the hybrid $L^1$--$L^2$ Gordin criterion~\eqref{eq:Hinv2} in the invertible setting.
In particular, $\BBW_\eps^{12}(1)\to_w \BBW^{12}(1)+ c$ for some $c\in\R$.

Then $\hx\to_w X$ in $D([0,\infty),\R^2)$ as $\eps\to0$ where $X$ is the solution to the SDE
\[
dX=\tilde a(X)\,dt + b(X)\,dW, \quad X(0)=\xi,
\]
with $\tilde a(X)=\left(\begin{array}{c} 0 \\ g(X)+ c
\end{array}\right)$.
\end{prop}

\begin{proof}
We have
$\hx^1(t)  =\xi^1+W^1_\eps(t)$ and
\begin{align*}
\hx^2(t) & =\xi^2+\eps^2\sum_{j=0}^{[t\eps^{-2}]-1}g(\x_j)
+\eps\sum_{j=0}^{[t\eps^{-2}]-1}(\x_j)^1 v^2(y_j) \\
& =\xi^2+\int_0^t g(x_\eps(s))\,dt
+\int_0^t \hx^1(s)\,dW^2_\eps(s)+A_\eps(t) \\
& =\xi^2+\int_0^t g(x_\eps(s))\,ds+\int_0^t (\xi^1+W^1_\eps(s))\,dW_\eps^2(s) +A_\eps(t)
\\
& =\xi^2+\int_0^t g(x_\eps(s))\,ds+\xi^1 W_\eps^2(t)+\BBW_\eps^{12}(t)+A_\eps(t)
\end{align*}
where
\begin{align*}
|A_\eps(t)| & \le \eps^2|g|_\infty + \eps|v|_\infty \max_{0\le j\le \eps^{-2}t} |(\x_j)^1|.
\\ & \le \eps^2|g|_\infty + 
\eps|v|_\infty \xi^1
+\eps|v|_\infty \max_{0\le s\le t} |W_\eps^1(s)|.
\end{align*}

In other words,
\[
\hx(t)=\xi+\int_0^t a(x_\eps(s))\,ds+ U_\eps(t)
\]
where
\[
U_\eps=\left(\begin{array}{c} W^1_\eps  \\ \xi^1 W^2_\eps + \BBW^{12}_\eps +A_\eps  \end{array}\right).
\]
The resulting solution map $\hx=\cG(U_\eps)$ is continuous on $C([0,K],\R^2)$ for all $K>0$ since $a$ is Lipschitz.
By the iterated WIP, $U_\eps\to_w U$ where
\[
U=\left(\begin{array}{c} W^1  \\ \xi^1 W^2 + \BBW^{12} \end{array}\right),
\qquad \BBW^{12}(t)=\int_0^t W^1\, dW^2+ t E^{12},
\qquad E^{12}=c,
\]
so the continuous mapping theorem shows that $\hx\to_w \cG(U)=X$
where
\[
dX= a(X)\,dt + dU, \quad X(0)=\xi.
\]
But $U^1=W^1$ and
\[
U^2(t)
=\int_0^t (\xi^1+ W^1)\, dW^2+  tE^{12}=
\int_0^t X^1\,  dW^2+ tE^{12}
=\int_0^t X^1\,  dW^2+ tc,
\]
yielding
$\tilde a(X)=\left(\begin{array}{c} 0 \\ g(X)+ c
\end{array}\right)$ as required.
\end{proof}

 \paragraph{Acknowledgements}
 We are grateful to the referee for helpful comments and suggestions.

\end{document}